\documentclass{amsart}
\usepackage{graphicx, amsfonts, amssymb, thm-restate, wasysym} 

\usepackage{etaremune}

\usepackage{hyperref}
\usepackage[%
    backend = biber,
    backref = true,
    style = alphabetic,
    citestyle = alphabetic,
    maxalphanames = 99,
    maxnames = 99,
    isbn = false,
    doi = false,
    url = false,
    ]{biblatex} 
\newbibmacro{string+doiurlisbn}[1]{%
 \iffieldundef{doi}{%
    \iffieldundef{url}{%
     \iffieldundef{isbn}{%
       \iffieldundef{issn}{%
         #1%
        }{%
          \href{http://books.google.com/books?vid=ISSN\thefield{issn}}{#1}%
        }%
      }{%
        \href{http://books.google.com/books?vid=ISBN\thefield{isbn}}{#1}%
      }%
    }{%
      \href{\thefield{url}}{#1}%
    }%
  }{%
    \href{http://dx.doi.org/\thefield{doi}}{#1}%
  }%
}

\DeclareFieldFormat{title}{%
  \usebibmacro{string+doiurlisbn}{\mkbibemph{#1}}}
\DeclareFieldFormat
  [article,inbook,incollection,inproceedings,patent,thesis,unpublished]
  {title}{\usebibmacro{string+doiurlisbn}{\mkbibquote{#1\isdot}}}
\DeclareFieldFormat
  [suppbook,suppcollection,suppperiodical]
  {title}{\usebibmacro{string+doiurlisbn}{#1}}

\AtEveryBibitem{%
  \clearfield{note}%
}

\theoremstyle{plain}
\newtheorem{thm}{Theorem}

\newtheorem{qst}[thm]{Question}
\newtheorem{lm}[thm]{Lemma}
\newtheorem{prp}[thm]{Proposition}
\theoremstyle{definition}

\declaretheorem[sibling=thm,name=Example,qed={\twonotes}]{xmp}
\declaretheorem[sibling=thm,name=Remark,qed={\eighthnote \twonotes}]{rmk}

\title{Balanced player arrangement in rotating court games}
\author{Alessandro Danelon}
\address{Department of Mathematics, University of Michigan, Ann Arbor, MI}
\email{\href{mailto:adanelon@umich.edu}{adanelon@umich.edu}}
\urladdr{\url{https://public.websites.umich.edu/~adanelon/}}
\date{27 September 2025}

\begin{document}





\begin{abstract}
    Consider a game involving a team with $n$ players, $k$ of which wear shirts marked with a letter $A$, while the others with a letter $B$, and such that only $s$ people play, while the remaining $n-s$ wait outside the court.
    At certain times the players rotate, and one player enters the court while another one leaves, each player keeping the same neighbors at all times.
    The house rules want at least $t$ players wearing a shirt with the letter $A$ in the court at each rotation. 
    We show that this can be achieved if and only if $nt\leq ks$.
    We use classical work on balanced words dating back to Christoffel and Smith.
\end{abstract}
\maketitle
\section{Introduction}

This problem arose during a volleyball tournament in which a team was composed of $3$ women and $7$ men.
There are $6$ people playing in the court while the remaining $4$ players await next to it.
At each rotation one player enters the court while another one leaves, each player keeping the same neighbors at all times.
Moreover, the house rules want at least $2$ women in the court after each rotation.
\begin{qst}
    Can we find a placement of players guaranteeing the house rules are observed?
\end{qst}

After generalizing the above setting, we find a criterion showing that in this setting it is not possible to play according to the house rules.

\subsection{Generalization}
For positive integers $k,s,n$ such that $k,s < n$, an $(n, k, s)$-\textit{configuration} is given by:
\begin{enumerate}
    \item a circle $C$ with $n$ spots, $k$ of which are filled with the letter $A$ and $n-k$ with the letter $B$, and
    \item the collection $\Sigma$ of subsets of $s$ consecutive spots on the circle.
\end{enumerate}
In the volleyball setting above, $n$ is the total number of players, $k$ is the number of players wearing an $A$-shirt, and $s$ is the number of players in the court.
The set $\Sigma$ contains all the possible combinations (given by the rotations) of players in the court for a given order of the players.
A configuration $C$ is $t$-\textit{admissible} if any element of $\Sigma$ contains at least $t$ letters $A$.
In the case above, we have $t = 2$, as the house rules ask for at least $2$ players with an $A$-shirt in the court.
Note that we can equivalently restate the problem asking that in any subset of consecutive $n-s$ players, there are at most $k-t$ players wearing a shirt with a letter $A$.

\subsection{Outline}
We pose the following main question:
\begin{qst}
    Is there a criterion for $t$-admissibility of a $(n, k, s)$-configuration?
\end{qst}

We answer positively the above question: a $(n, k, s)$-configuration is $t$-admissible if and only if $nt \leq ks$.
We first note that our circular arrangement of $n$-players corresponds to infinite words on two letters of period $n$ (see Section~\ref{ssec:balancedwords} for definitions).
A \textit{mechanical word} of slope $k/n$ is a infinite word of period $n$ such that every subset of $m$ consecutive spots contains at least $\left \lfloor \frac{k}{n}\cdot m \right \rfloor$ and at most $\left \lceil\frac{k}{n}\cdot m \right \rceil$ letters $A$.
We use this fact to show the existence of a $t$-admissible configuration when $nt\leq ks$.
The fact that there are not $t$-admissible configurations when $nt>ks$ follows from a pigeonhole-type argument.
All the details are in Proposition~\ref{prp:main}.

\subsection{Connection to existing literature}
The problem can be restated in the setting of classical discrepancy theory.
Consider a ground set $X =\{x_1, \dots, x_n\}$ and the collection $\Sigma$ of subsets of the form $\{x_{i}, \dots, x_{i+m}\}$ where the indices are taken modulo $n$ when they exceed $n$.
Consider functions $\chi: X \to \{-1, 1\}$ representing the possible colorings of $X$ with two colors and, for a subset $S$ of $X$, define $\chi(S) = \sum_{x\in S}\chi(x)$.
The smaller this value is, the more balanced is the coloring $\chi$ of $S$.
The {\em discrepancy} of a coloring for $X$ with respect to $\Sigma$ is given by
\[
\operatorname{Disc}(X) = \min_{\chi}\max_{S\in \Sigma} |\chi(S)|.
\]
Our results finds a coloring $\chi$ guaranteeing that $\operatorname{Disc}(X) \leq m - 2 \left \lfloor \frac{m}{n/k} \right \rfloor$.
There are more general theorems working for any collection $\Sigma$.
The Beck–Fiala theorem for this setting implies that $\operatorname{Disc}(X)\leq 2m -1$, while the Spencer upperbound guarantees $\operatorname{Disc}(X) = O\left ( \sqrt{n \log(2)}\right)$ (\cite[Section~4.1]{discrepancy}).
Moreover, \cite{spencer:six} guarantees $6\sqrt{n}$.

Our solution revolves around the construction of mechanical words that are examples of Christoffel/Sturmian word.
Vast literature investigates this subject that has connections from number theory to theoretical computer science.
See, for example, \cite{sturmianwords,combinatoricsonwords, algcombonwords, appcombonwords, automaticseq, vuillon:balancedwords, wehlou}.
We also provide an algorithm outputting a mechanical word via the use of continued fractions.
We discovered at a later stage that the connection between continued fraction and balanced words was first made in \cite{smith:continuedfractions}.

\subsection{Acknowledgments}
Thanks to Andrew Snowden for useful discussions, to Zhe Su for wondering about a general approach to tackle this problem, and to the members of my volleyball team ``Kiss my Ace" for proposing this problem.

\section{Balanced words and proof}\label{ssec:balancedwords}
For this section we refer to \cite[Section~1.2.2 and Section~2.1.2]{algcombonwords} for definitions.
Let $\mathcal{A} = \{0, 1\}$ be the alphabet and let $\mathcal{A}^*$ be the set of (finite) words with entries in $\mathcal{A}$ together with the empty word $\epsilon$.
Let $x, y$ be finite words, the operation of concatenation $x*y$ turns $\mathcal{A}^*$ into a monoid.
Infinite words (one-sided) are elements of $\mathcal{A}^\mathbb{N}$.
An infinite word $w$ is \textit{periodic} of period $n$ if there exists a $z \in A^* \setminus\{\epsilon\}$ with $n$ letters such that $w = z*z*z*\cdots$.
Let $w(n)$ be the $n$-th letter of a word $w$, where $n\geq 0$ and $w(-1) = \epsilon$.
Let $w(n, n + k)$ be the \textit{factor} of length $k+1$ of $w$ starting at the letter $w(n)$ and ending at the letter $w(n + k)$.
The \textit{weight} $h(w)$ of a finite word $w$ is the number of $1$s in $w$.

Our main definition is the following: a \textit{mechanical word} $w$ of \textit{slope} $0 < \alpha \leq 1$ is an infinite word defined by $w(n) = \left \lceil \alpha (n+1) \right \rceil - \left \lceil \alpha  n \right \rceil$
for all $n \geq 0$.
\begin{lm}\label{lm:main}
    Let $w$ be a mechanical word of rational slope $\alpha = \frac{k}{n}$, then it is periodic, and for any factor $u$ of length $m$ of $w$ we have
    \[ 
    \lfloor m\alpha \rfloor \leq h(u) \leq \lceil m\alpha \rceil.
    \]
\end{lm}
\begin{proof}
    This is \cite[Lemma~2.1.14]{algcombonwords}.
    By definition it is periodic of period $n$.
    Suppose that $u = w(n_1, n_2) = w(n_1, n_1 + m - 1)$, then $h(u) = h(w(0, n_2)) - h(w(0, n_1-1)) = \lceil \alpha (n_1 + m)\rceil - \lceil \alpha n_1\rceil$.
    Then, $\alpha m - 1 < h(u) < \alpha m + 1$, hence $\lfloor m\alpha \rfloor \leq h(u) \leq \lceil m\alpha \rceil$.    
\end{proof}

The following is the proof of the admissibility criterion.

\begin{prp}\label{prp:main}
    There exists a $t$-admissible configuration if and only if $nt \leq ks$.
\end{prp}
\begin{proof}
Suppose $ks < nt$, and let $C$ be a configuration.
Order the letters $A$ clockwise. We can start at any $A$.
Let $O$ be an element of $\Sigma$ and let $a_{i,j}$ be the rotation distance (how many rotations clockwise) between the $i$-th $A$ and the $j$-th elements of $O$.
The $a_{i,j}$ take values between $0$ and $n-1$.
Note that the $i$-th $A$ belongs to $O$ if and only if $a_{i,j} = 0$ for some $j =1, \dots, s$.
The integers $a_{i,j}$ form a $k\times s$-matrix where the $i$-th row is the vector $(a_{i,1}, \dots, a_{i,s})$ with no repeated values, and one can read off how many letters $A$ belong to $O$ by counting the zeroes of this matrix.
By the pigeonhole principle there must be a value $a$ between $0$ and $n-1$ that is attained at most $ks/n < t$ times.
Then, the subset $O'$ obtained by the rotation of $O$ by $a$ spots shows that $C$ is not admissible.

Viceversa, assume that $ks \geq nt$, construct a (infinite) mechanical word of slope $\frac{k}{n}$, and substitute the letter $A$ for the number $1$ and the letter $B$ for the number $0$.
Since $k/n$ is rational, the infinite mechanical word is equivalent to a circular arrangement of $n$ letters.
Consider any set $O$ of $s$ consecutive spots.
By Lemma~\ref{lm:main}, $O$ contains at least $\left \lfloor \frac{ks}{n}\right \rfloor$ letters $A$.
Our condition implies that $t \leq \frac{ks}{n}$, and, since it is an integer $t \leq \left \lfloor \frac{ks}{n}\right \rfloor$, so $O$ contains at least $t$ letters $A$.
\end{proof}


\section{Equivalent algorithms}
In this section we describe a way to arrange $k$ letters $A$ and $(n-k)$ letters $B$ on a circle so that any section of $m$ consecutive players contains roughly the same number of $A$ and $B$.
This algorithm was our original solution to the problem using the continued fraction expansion of $\frac{n}{k}$.
At a later stage we found a connection with the algorithm of \cite{smith:continuedfractions}, and hence with mechanical words.
We show how the two algorithms are equivalent and how they output the same mechanical word.
We think it is still worth mentioning it here as it is another way for assigning spots to each player on the fields.

\subsection{Description of the algorithm}\label{ssec:myalgo}

Given integers $n$ and $k$, we want to dispose the players on a circle in such a way that any two consecutive letters $A$ are the furthest apart from each other.
In order to do so, we follow the steps below.
\begin{enumerate}
    \item Perform the Euclidean algorithm on $n$ and $k$.
    We use the following notation:
    \[
    r_{j-2} = q_j r_{j-1} + r_j,
    \]
    with $r_j < r_{j-1}$, $r_{-3} = n$ and $r_{-2} = k$.
    We end at the step $i$ when $r_{i+1} = 0$, so we get:
\begin{align*}
    n = \; & \; r_{-3} = q_{-1}r_{-2} + r_{-1}\\
k = \;& \; r_{-2} = q_0 r_{-1} + r_0\\
& \; r_{-1} = q_1 r_0 + r_1\\
&\; \; \; \; \; \; \; \; \,\vdots \\
&r_{i-3} = q_{i-1} r_{i-2} + r_{i-1}\\
&r_{i-2} = q_i r_{i-1} + r_i\\
&r_{i-1} = q_{i+1}r_i.\\
\end{align*}
    \item The minimal value for the index $i$ is $-2$.
    In this case $k$ divides $n$ and we construct a configuration with one letter $A$ each $(q_{-1}-1)$ consecutive letters $B$.
    \item For $i > -2$ there is a nonzero remainder of the division of $n$ by $k$.
    Before placing the letters, we create a sequence with symbols $-$ and $+$.
    We write the sequence clock-wise on a circle as follows.
    \begin{enumerate}
        \item Put $r_i$ symbols $+$, and add $q_{i+1}-1$ symbols $-$ between two consecutive $+$.
        We have now a sequence with $r_{i-1}$ symbols, $r_i$ of which are $+$.
        \item Add a symbol $-$ after each of the $r_i$ symbols $+$.
        Turn the previous $r_{i-1}-r_i$ symbols $-$ into symbols $+$.
        We now have a sequence with $r_i + r_{i-1}$ symbols: and $r_{i-1}$ are symbols $+$ and $r_i$ are symbols $-$.
        \item Between two consecutive symbols $+$ in the sequence, add $q_i - 1$ symbols $-$.
        The sequence now has $r_{i-2}$ symbols.
        
        \item Iterate steps $(b)$ and $(c)$ above until reaching $r_{-2}$ total symbols, $r_{-1}$ of which are $+$.
    \end{enumerate}
    \item At this point we insert the letters: after each $+$ we add a letter $B$, we turn all the $+$ and $-$ into letters $A$ (we have $r_{-2} = k$ of them), and we add $(q_{-1}-1)$ letters $B$ in between two consecutive letters $A$.
\end{enumerate}

At the end of the algorithm, we filled all the $n$-spots of the circle.

\begin{xmp}
    Below we run the algorithm for $n = 23$ and $k = 10$.
    \begin{enumerate}
        \item Compute $23 = 2\cdot 10 + 3$, $10 = 3 \cdot 3 + 1,$ and $3 = 3\cdot 1 + 0$.
        \item False: $i = 0$.
        \item True.
        \begin{enumerate}
            \item Output $[+, -, -]$.
            \item Output $[+, -, +, +]$
            \item Output $[+, -, -, -, +, -, -, +, -, -]$.
            \item Done.
        \end{enumerate}
        \item Output $[ABB AB AB AB ABB AB AB ABB AB AB]$.
    \end{enumerate}
\end{xmp}

\begin{xmp}
    Below there is another example with one more iteration of the above algorithm.
    Suppose we have the recurrences:
    \begin{enumerate}
        \item $ r_{-3} = n = 87 = 2*36 + 15 = q_{-1}r_{-2} + r_{-1}$
        \item $ r_{-2} = k = 36 = 2*15 + 6 = q_{0}r_{-1} + r_{0}$
        \item $ r_{-1} = 15 = 2*6 + 3 = q_{1}r_{0} + r_1$
        \item $ r_{0} = 6 = 2*3 = q_{2}r_1$
    \end{enumerate}
    Then the algorithm performs the following steps (the symbols $``|''$ are for (the author's) visual help):
    \begin{etaremune}
        \item $[+,-|+,- | +, -]$
        \item $[+,-, -, +, - |+, -, -, +, - | +, -, -, +, -]$
        \item $[+, -, -, +, -, +, -, +, -, -, +, - |\\ +, -, -, +, -, +, -, +, -, -, +, - |\\ +, -, -, +, -, +, -, +, -, -, +, -]$
        \item $[ ABBABAB \text{ } ABBAB \text{ } ABBAB \text{ } ABBABAB \text{ } ABBAB | \\
        ABBABAB \text{ } ABBAB \text{ } ABBAB \text{ } ABBABAB \text{ } ABBAB |\\
        ABBABAB \text{ } ABBAB \text{ } ABBAB \text{ } ABBABAB \text{ } ABBAB]$
    \end{etaremune}
\end{xmp}

\begin{rmk}
    \sloppy
    If $n$ and $k$ are coprime (so that $r_i = 1$), let the sequence $[q_{-1}, q_0, q_1, \dots, q_{i+1}]$ be the continued fractions quotients of the division $n$ by $k$.
    The recurrence sequence $a_j = q_{j}a_{j-1} + a_{j-2}$ with $a_{i-1} = 0$, $a_{i} = 1$, and starting from $j = i + 1$ retrieves the steps in the Euclidean algorithm bottom-up ($a_{-1} = n$ and $a_0 = k$). 
\end{rmk}

\subsection{Smith's construction}

For this section we refer to \cite{smith:continuedfractions}, whose original motivation was to order the points $P = 1/p, 2P = 2/p, \dots, pP = p/p$, and $Q = 1/q, 2Q = 2/q, \dots, qQ = q/q$ on $[0,1]$ for two coprime integers $p,q$.
In our setting, that construction runs as follows.
Suppose $p>q$, and let $[\mu_1, \dots, \mu_t]$ be their continued fraction expansion, namely:
\[
p/q = \mu_1 + \cfrac{1}{\mu_2 + \cfrac{1}{\cdots + \cfrac{1}{\mu_{k-1}+\cfrac{1}{\mu_{t}}}}}.
\]
Construct the following recursive sequence:

\begin{enumerate}
    \item $S_1 = B^{\mu_1}A$,
    \item $S_2 = S_1^{\mu_2}B$,
    \item $S_3 = S_2^{\mu_3}S_1$,
    \item $\cdots$
    \item $S_t = S_{t-1}^{\mu_t}S_{t-2}$,
\end{enumerate}

Arranging the above ordered letters on a circle gives a circular arrangement of players.

\begin{prp}
    Let $n > k $ be two coprime integers with continued fraction expansion $[\mu_1, \dots, \mu_t]$.
    Then the Smith's algorithm on the continued fraction expansion $[\mu_1-1, \dots, \mu_t]$ and the algorithm of Section~\ref{ssec:myalgo} on $n$ and $k$ output the same arrangement (up to rotation).
\end{prp}

\begin{proof}
    We run the Smith's sequence bottom up and find the correspondence with the algorithm of Section~\ref{ssec:myalgo}.
    For this proof we refer to the latter simply as the algorithm.
    Recall that $n$ and $k$ are coprime, and note that $\mu_t = q_{i+1}, \mu_{t-1} = q_i,$ and so on.
    The first step of Smith's algorithm correspond to run the algorithm up to  step $3(b)$.
    The $S_{t-1}$ correspond to the symbols $+$ (so there are ${\mu_t}$ of them) and $S_{t-2}$ is the symbol $-$.
    Notice that the above configuration is shifted with respect to the algorithm by a rotation of two steps clockwise:
    \[
     S_{t-1}
    \cdots S_{t-1} S_{t-2} \sim (+ \cdots + - ) \sim (+-+\cdots +),
    \]
    where we use the symbol $\sim$ for the equivalence between the expressions and the round-brackets single out the output of the algorithm.
    
    Rewriting $S_{t-1}$ in terms of $S_{t-2}$ and $S_{t-3}$ corresponds to perform steps $3(c)$ and $3(b)$ in the second iteration.
    Indeed, the algorithm substitutes for each $+$ a sequence $+-+\cdots+$ where the tail has $\mu_{t-1}-1$ symbols $+$, while the symbol $-$ becomes a $+$.
    Rewriting $S_{t-2}$ for the symbols $+$ and $S_{t-3}$ for the symbols $-$ shows that one gets the same output of Smith's construction (again up to a rotation by two steps counterclockwise).
    At the end, we have a sequence of $S_2$ and $S_1$, where each $S_2$ corresponds to a $+$ and each $S_1$ to a $-$.
    Finally, rewriting $S_2$ corresponds to the step $3(c)$ and the first part of step $(4)$: we get to  $+^{\mu_2} B$.
    Rewriting  $S_1$ correspond to exchange each $+$ for $B^{\mu_1-1}A$, as in the last step of the algorithm.    
\end{proof}

\subsection{Equivalence with mechanical words}

\begin{prp}\cite{smith:continuedfractions}\cite[Theorem~7.2]{combinatoricsonwords}
    Let $S_t$ be the expression of the Smith procedure on $n$ and $k$ coprime.
    The expression $(A*S_{t}(0, n-2)*B)^{*\mathbb{N}}$ is a mechanical word of slope $k/n$.
\end{prp}

\begin{proof}
    Use \cite[Lemma~7.2]{combinatoricsonwords}.
\end{proof}

This shows that the algorithm outputs an admissible configuration when $n$ and $k$ are coprime.
If not, namely, $r_i > 1$, by construction, we have $r_i$ sections with $n' = n/r_i$ elements each, and containing $k' = k/r_i$ letters $A$.
Fix a subset $O$ with $m$ consecutive letters, and write $m = cn' + m'$ for some integer $c$.
By assumption the set $O$ contains at most $ck' + \left \lceil\frac{m'}{n/k} \right\rceil$ and at least $ck' + \left \lfloor\frac{m'}{n/k} \right \rfloor$ letters $A$.
These numbers coincide with $\left \lceil \frac{m}{n/k} \right \rceil$ and $\left \lfloor \frac{m}{n/k} \right \rfloor$.

\printbibliography
\end{document}